\documentclass[12pt,a4paper,oneside]{amsart}
\usepackage{color,amssymb,latexsym,amsfonts,textcomp,fancyhdr,calc,graphicx}
\usepackage{amsmath,amstext,amsthm,amssymb,amsxtra}
\usepackage[left=3cm,right=3cm,bottom=2cm,top=2cm]{geometry} 
\usepackage[colorlinks,citecolor=red,pagebackref,hypertexnames=false]{hyperref}
\usepackage{dsfont}
\usepackage[framemethod=tikz]{mdframed}
\usepackage[inline]{asymptote}

\usepackage[colorlinks,citecolor=red,pagebackref,hypertexnames=false]{hyperref}

\usepackage{txfonts,pxfonts,tikz} 
\usepackage{tgschola}
\usepackage[T1]{fontenc}

\numberwithin{equation}{section}

\theoremstyle{plain}
\newtheorem{theorem}{Theorem}[section]
\newtheorem{lemma}[theorem]{Lemma}

\newtheorem{question}{Question}

\theoremstyle{definition}
\newtheorem{definition}[theorem]{Definition}

\newtheorem{case[theorem]}{Case}

\theoremstyle{remark}
\newtheorem{remark}[theorem]{Remark}

\numberwithin{equation}{section}

\newcommand{\beql}[1]{\begin{equation}\label{#1}}
\newcommand{\eeq}{\end{equation}}
\newcommand{\comment}[1]{}

\newcommand{\Abs}[1]{{\left|{#1}\right|}}

\newcommand{\Ceil}[1]{{\left\lceil{#1}\right\rceil}}

\newcommand{\Set}[1]{{\left\{{#1}\right\}}}

\newcommand{\RR}{{\mathbb R}}

\newcommand{\ZZ}{{\mathbb Z}}
\newcommand{\NN}{{\mathbb N}}

\newcommand{\QQ}{{\mathbb Q}}

\newcommand{\one}{{\mathds{1}}}

\newcommand{\supp}{{\rm supp\,}}

\newcommand{\vol}{{\rm vol\,}}
\newcommand{\ft}[1]{\widehat{#1}}

\newcommand{\diam}{{\rm diam\,}}
\newcommand{\co}{{\rm co\,}}

\newcounter{rem}
\makeatletter
\@namedef{subjclassname@2020}{\textup{2020} Mathematics Subject Classification}
\makeatother

\begin{document}

\title[Functions tiling with several lattices]{Functions tiling with several lattices}

\author{Mihail N. Kolountzakis}
\address{Department of Mathematics and Applied Mathematics, University of Crete, Voutes Campus, 70013 Heraklion, Crete, Greece. \url{https://orcid.org/0000-0001-9598-4580}}
\email{kolount@gmail.com}

\author{Effie Papageorgiou}
\address{Department of Mathematics and Applied Mathematics, University of Crete, Voutes Campus, 70013 Heraklion, Crete, Greece.
\url{https://orcid.org/0000-0001-8475-3226}}
\email{papageoeffie@gmail.com}


\date{June 22, 2021}

\thanks{Supported by the Hellenic Foundation for Research and Innovation, Project HFRI-FM17-1733 and by grant No 4725 of the University of Crete.}

\begin{abstract}
We study the problem of finding a function $f$ with ``small support'' that simultaneously tiles with finitely many lattices $\Lambda_1, \ldots, \Lambda_N$ in $d$-dimensional Euclidean spaces. We prove several results, both upper bounds (constructions) and lower bounds on how large this support can and must be. We also study the problem in the setting of finite abelian groups, which turns out to be the most concrete setting. Several open questions are posed.
\end{abstract}

\maketitle

\tableofcontents

\setlength{\parskip}{0.5em}

\section{Introduction to the Steinhaus tiling problem and its variants}

The classical Steinhaus tiling problem concerns tiling by translations only. This is the version of tiling where we have a set (the tile), say in the plane, and we are {\em translating} it around in such a way that every point of the plane is covered exactly once by these translates. In the Steinhaus tiling problem we are seeking a tile that can tile {\em simultaneously} with many different sets of translations. The most important case is: can we find a subset of the plane which can tile (by translations) with all rotates of the integer lattice $\ZZ^2$?

There are two major variations of the Steinhaus problem: the \textit{measurable} and the \textit{set-theoretic} case. In the measurable case we demand our tile to be a Lebesgue measurable subset of $\RR^d$ and we are, at the same time, relaxing our requirements and are allowing a subset of measure 0 of space not to be covered exactly once by the translates of the tile. In the set-theoretic case we allow the tile to by any subset and we typically ask that every point is covered exactly once, allowing no exceptions.

In this paper we are dealing with the measurable Steinhaus tiling problem. Here the tile is not necessarily a set but can be a measurable function. Tiling now means that this function $f:\RR^d \to \RR$ is translated around by a countable set of translates $T \subseteq \RR^d$ and these copies add up to a constant, almost everywhere
$$
\sum_{t \in T} f(x-t) = \mathrm{const.},\ \ \ \text{for almost every $x \in \RR^d$.}
$$
Clearly tiling by a function is a generalization of tiling by a set.

Satisfying the requirements of the Steinhaus tiling problem with a function instead of with a set is generally much easier. The problem becomes interesting only if one asks for further properties that this function should have. Therefore we strive to find a function with {\em small support}, or to prove that the support of such a function must necessarily be large. Asking for $f$ to have a small support goes against $f$ having the ability to tile space, especially with nany different sets of translations $T$. The reason is that for $f$ to tile by translations with $T$ its Fourier transform must contain a rich set of zeros \cite{kolountzakis2004milano}. This set of zeros must be able to support the Fourier transform of the measure $\delta_T = \sum_{t \in T} \delta_t$ (which encodes the set of translations). By the well known uncertainty principle in harmonic analysis a rich set of zeros for $\ft{f}$ usually requires (in various different senses) a large support for $f$ \cite{havin1994uncertainty}.

\subsection{Previous results}

Komj\'ath \cite{komjath1992lattice} answered the Steinhaus question in the affirmative in $\RR^2$ when tiling by all rotates of the set $B=\ZZ\times\Set{0}$
showing that there are such Steinhaus sets (but such a set $A$ cannot be measurable as was shown
recently in \cite{kolountzakis2017measurable}).
Sierpinski \cite{sierpinski1958probleme} showed that a bounded set $A$ which is either closed or open cannot have the
lattice Steinhaus property (that is, intersect all rigid motions of $\ZZ^2$ at exactly one point -- another way to say that $A$ tiles precisely with all rotates of $\ZZ^2$).
Croft \cite{croft1982three} and Beck \cite{beck1989lattice} showed that no bounded and measurable set $A$ can
have the lattice Steinhaus property (but see also \cite{mallinikova1995}).
Kolountzakis \cite{kolountzakis1996problem,kolountzakis1996new} and Kolountzakis and Wolff \cite{kolountzakis1999steinhaus}
proved that any measurable set in the plane that has the measurable Steinhaus property must necessarily
have very slow decay at infinity (any such set must have measure 1). In \cite{kolountzakis1999steinhaus}
it was also shown that there can be no measurable Steinhaus sets in dimension $d\ge 3$ (tiling with all rotates $\rho\ZZ^d$, where $\rho$ is in the full orthogonal group)
a fact that was also shown later by Kolountzakis and Papadimitrakis \cite{kolountzakis2002steinhaus} by a very different method.
See also \cite{chan2007steinhaus,mauldin2003comments,ciucu1996remark,srivastava2005steinhaus}. Kolountzakis \cite{kolountzakis1997multi} looks at the case where
we are only asking for our set to tile with {\em finitely many} lattices, not all rotates as in the original problem, which we are also doing in this paper.
A very important strengthening of some of the results in \cite{kolountzakis1997multi} was given in \cite{han2001lattice}, where it was proved that a common tile for {\em two} lattices in Euclidean space always exists, and a very nice application of this fact was given to constructing a Gabor orthogonal basis for $\RR^d$ for a given lattice $K$ of translations and a given lattice $L$ of modulations, subject only to the necessary condition that $\det{K} \det{L} = 1$.
In a major result Jackson and Mauldin \cite{jackson2002sets,jackson2002lattice} proved the existence of Steinhaus sets in the plane which tile with all rotates of $\ZZ^2$
(not necessarily measurable). Their method does not extend to higher dimension $d \ge 3$. See also \cite{mauldin2001some,jackson2003survey}.
It was also shown in \cite{kolountzakis2017measurable} that a set $A$ which tiles with all rotates of a finite set $B$ cannot be measurable.

\subsection{Structure of this paper}

In \S\ref{sec:lattices} we study the common tiles $f$ for a finite collection of $N$ lattices in $\RR^d$. The emphasis is on the dependence of the size of the common tile on $N$.

In \S\ref{sec:convolution-tiles} we study the diameter of common tiles, a study that had begun in \cite{kolountzakis1999steinhaus}. We show in Theorem \ref{th:long-common-tiles} that there are cases of lattices for which we have a linear lower bound (matching the linear upper bound arising from convolution tiles) for the diameter of the support of any common tile.

In \S\ref{sec:small-volume} we study the volume of the common tiles instead of their diameter. In Theorem \ref{th:volume-lb} we show that convolution tiles with nonnegative convolution factors necessarily have volume $\gtrsim N^d$ for their support. In Theorem \ref{th:1d-lb} we give a simple non-trivial lower bound for the length for a convolution tile of two factors.

Then in \S\ref{sec:many-relations} we turn away from lattices satisfying genericity conditions (e.g. having a direct sum) to lattices which satisfy many algebraic relations. We show that this helps greatly with the diameter and volume of the support of their common tiles (construction in Theorem \ref{th:many-relations} for $d \ge 2$ and in Theorem \ref{th:many-relations-1d} for $d=1$).

In \S\ref{sec:no-measurability} we study common tiles of lattices without the measurability assumption. For two lattices of the same volume in $\RR^d$ these were known to always exist \cite{kolountzakis1997multi}. We show in Theorem \ref{th:equal-volume} that the condition of equal volume is actually necessary for the existence of a common fundamental domain, which had remained an open question in \cite{kolountzakis1997multi}.

In \S\ref{sec:finite-groups} we study the problem of finding a common tiling function for subgroups of a finite abelian group, trying now to minimize the size (cardinality) of its support.

In Theorem \ref{th:no-intersection} we show that we can reduce the problem to the case where the two subgroups have no intersection, i.e. we reduce the study to direct products of groups. Theorem \ref{th:multiple} solves the problem exactly in the special case when the size of one group divides the size of the other. We then connect the problem to the problem of the support of {\em copulas} (a generalization of doubly stochastic matrices to the non-square case) as they have been studied in combinatorics and statistics and we deal with some special cases in Lemma \ref{lm:r}\footnote{This problem has since been solved completely in \cite{loukaki2022doubly}}.

Several open questions are posed throughout.

\section{The diameter and volume of soft multi-lattice tiles}\label{sec:lattices}

\subsection{Convolution tiles and their diameter}\label{sec:convolution-tiles}
It has long been known \cite{beck1989lattice,kolountzakis1996problem} that for a function $f \in L^1(\RR^d)$ to tile with all rotates of $\ZZ^d$ it is necessary and sufficient that $\ft{f}$ vanishes on all spheres centered at the origin that contain any integer lattice point. This easily implies that for $d \ge 2$ any such function must have unbounded support and even more  quantitative lower bounds for the the rate of decay of $f$ near infinity \cite{kolountzakis1996problem,kolountzakis1996new,kolountzakis1999steinhaus}.

This is no longer true if one restricts the number of rotates of $\ZZ^d$ that we demand $f$ tiles with. It makes sense to generalize the question and ask for a function $f$ which tiles with a given set of lattices $\Lambda_i \subseteq \RR$, $i=1, 2, \ldots, N$, and such that
$$
\diam\supp f
$$
is small. We remind \cite{kolountzakis1999steinhaus} that for $f$ to tile with the lattices $\Lambda_i$ it is necessary and sufficient for $\ft{f}$ to vanish on the dual lattices $\Lambda_i^*$, except at 0. (The dual lattice of a lattice $\Lambda = A \ZZ^d$, width $A$ being a non-singular $d \times d$ matrix, is the lattice $\Lambda^* = A^{-\top}\ZZ^d$.)

Therefore we have to find a function $f$, with support of small diameter, such that $\ft{f}$ is 0 on each $\Lambda_i^*\setminus\Set{0}$.
The first thing that comes to mind is to take $f$ to be a convolution. It takes a moment to verify that if $f$ tiles with a set of translates $T$ then so does $g*f$ for any $g \in L^1(\RR^d)$. One can either verify this by checking the definition of tiling for $g*f$ or observe that tiling is a condition that can be checked on the Fourier side \cite{kolountzakis2004milano} and $\ft{g*f} = \ft{g} \cdot \ft{f}$ has an even richer set of zeros that $\ft{f}$.

So, since $\ft{f}$ has to vanish on the dual lattices $\Lambda_i^*\setminus\Set{0}$ we can take
\beql{convolution}
f = \one_{D_1} * \one_{D_2} * \cdots * \one_{D_N},
\eeq
where $D_i$ is a fundamental parallelepiped of $\Lambda_i$. Since $D_i + \Lambda_i$ is a tiling it follows that $\widehat{\one_{D_i}}$ vanishes on $\Lambda_i^*\setminus\Set{0}$ and that $f$ vanishes on their union and hence tiles with all $\Lambda_i$. This can be slightly generalized by taking, instead if the indicator functions $\one_{D_i}$ any function $f_i$ that tiles with $\Lambda_i$
\beql{convolution-of-functions}
f = f_1 * f_2 * \cdots f_N.
\eeq

The following observation was already made in \cite{kolountzakis1999steinhaus} in the case $f_i = \one_{D_i}$.
\begin{theorem}\label{th:convolution-lb}
If $\Lambda_1, \ldots, \Lambda_N$ are lattices in $\RR^d$ of volume $c_1 \le \vol\Lambda_i$, $f_i$ tiles with $\Lambda_i$ at some non-zero level and $f = f_1 * f_2 * \cdots * f_N$ then
\beql{convolution-lower-bound}
\diam\supp f \ge C_d N.
\eeq
\end{theorem}

\begin{proof}
By the pigeonhole principle on can pick a coordinate axis, say the first one, such that at least $N/d$ of the functions $f_i$ have $\supp f_i$ project onto the first coordinate axis on a set of diameter $\ge a_d$, where $a_d$ is a constant that depends on $d$ only. For if not then we would be able to find a $f_i$ whose support is contained in a cube of arbirarily small side. Since $f_i$ tiles with $\Lambda_i$ its support has to contain at least one element from almost all of the cosets of $\Lambda_i$ in $\RR^d$. This contradicts our assumed lower bound on $\vol\Lambda_i$.

The Tichmarsh convolution theorem \cite{donoghue2014distributions} says that
$$
\co\supp(A*B) = \co\supp A + \co\supp B,
$$
where $\co$ denotes the convex hull of a set and $A, B$ are two arbitrary integrable functions of compact support.

If we write $\phi$ for the convolution of those $f_i$ that we collected in the first paragraph of this proof and $\psi$ for the remaining $f_i$ then $f = \phi*\psi$ and
$$
\co\supp f = \co\supp \phi + \co\supp \psi,
$$
which implies that
\beql{diamsuppf}
\diam\supp f = \diam\co\supp f \ge \diam\co\supp \phi \ge \diam \pi_1 \co \supp \phi,
\eeq
where $\pi_1$ denotes projection onto the first coordinate axis.

Again by the Titchmarsh convolution theorem, $\pi_1\co\supp\phi$ is the sum of the $\pi_1\co\supp f_i$ for those $f_i$ that participate in the definition of $\phi$ and for these we know that
$$
\diam\pi_1\co\supp f_i \ge \diam \pi_1 \supp f_i \ge a_d.
$$
But for any two one-dimensional sets $E, F \subseteq \RR$ we have $\diam(E+F) = \diam E + \diam F$, which implies, using \eqref{diamsuppf}, that
$$
\diam\supp f \ge \frac{a_d}{d} N,
$$
as we had to prove.

\end{proof}


If the lattices $\Lambda_i$ satisfy some ``roundness'' assumption, e.g. if each $\Lambda_i$ is assumed to have a fundamental domain of diameter bounded independent of $N$ (as in the important case when all the lattices are rotates of $\ZZ^d$), then the convolution tile \eqref{convolution} clearly has diameter which is also at most $C\cdot N$.

On the other hand we have the following rather general lower bound for the diameter of the support \cite{kolountzakis1999steinhaus} assuming only a certain genericity assumption \eqref{no-intersection} on the $\Lambda_i$.
\begin{theorem}\label{kw-lb}
If $\Lambda_1, \ldots, \Lambda_N \subseteq \RR^d$, $d \ge 1$, are lattices of volume equal to 1 such that
\beql{no-intersection}
\Lambda_i \cap \Lambda_j = \Set{0}\ \  \text{ for all $i \neq j$,}
\eeq
then if $f$ tiles with all these lattices we have
\beql{lb}
\diam\supp f \ge C_d N^{1/d}.
\eeq
\end{theorem}


The main question is therefore:
\begin{question}
Can the gap between the lower bound \eqref{lb} and the linear upper bound $O(N)$ achievable by the convolution tile \eqref{convolution} (in the case of ``round'' lattices, having fundamental domains bounded in diameter by a constant) be bridged?

Are there examples of lattices $\Lambda_i$, $i=1,2,\ldots,N$, satisfying \eqref{no-intersection} and a non-zero function $f \in L^1(\RR^d)$ that tiles with all $\Lambda_i$ and such that
$$
\diam\supp f = o(N)?
$$

In other words, do there exist collections of lattices for which a common tile $f$ can be found which is diameter-wise more efficient than the convolution construction \eqref{convolution}?
\end{question}

Our next result is that for some collections of lattices the linear upper bound cannot be improved.
The lattices given are both ``round'' (have a fundamental domain bounded independent of $N$) and satisfy the genericity assumption \eqref{no-intersection}. There are however collinearities so, in some sense, this is not a generic situtation. 
\begin{theorem}\label{th:long-common-tiles}
For $d \ge 1$ and for each $N$ there are lattices $\Lambda_1, \ldots, \Lambda_N \subseteq \RR^d$, of volume 1, such that if $f \in L^1(\RR^d)$, $\int f \neq 0$, tiles with all of them then
$$
\diam\supp f \ge C_d N.
$$
\end{theorem}
\begin{proof}
We give the proof in the case $d=2$. It works with obvious changes in all dimensions $d > 2$ and it is even easier in dimension $d=1$.

Take $\Lambda_i^*$ to be generated by the two vectors
$$
u_i = (0, a_i), v_i = (1/a_i, 0),
$$
where the numbers $a_1, \ldots, a_N$ are linearly independent over $\QQ$ and
$$
0.9 < a_i < 1.
$$
If $f$ tiles with all $\Lambda_i$ then $\ft{f}$ vanishes on all points of the form
$$
(0, k \cdot a_i),\ \ i=1, 2, \ldots, N,\ \ k \in \ZZ\setminus\Set{0}.
$$
Since all these points are different it follows that the density of zeros on the $y$-axis is $\ge C \cdot N$. This implies that
$$
\diam\supp \pi_2(f) \ge C \cdot N
$$
(say, by Jensen's formula) where $\pi_2(f)$ is the one-variable function
$$
\pi_2(f)(y) = \int_{\RR} f(x, y) \,dx.
$$
(This is not an identically zero function by our assumption on the integral of $f$.)
This is turn implies
$$
\diam\supp f \ge C \cdot N.
$$

\end{proof}

\begin{remark}
Let us point out here that, on the real line at least, the following trick can shorten the diameter of a convolution tile by 1, when tiling simultaneously by two integer lattices. For instance, suppose we want to find a common tile in $\RR$ for the lattices $\Lambda_1 = m\ZZ$ and $\Lambda_2 = n\ZZ$, with $m, n \in \NN$. If we take the convolution of the individual tiles $\one_{[0, m]}$ and $\one_{[0, n]}$ then we obtain the commong tile
$$
f_1 = \one_{[0, m]} * \one_{[0, n]},
$$
which has diameter $m+n$.

But we can also work first in $\ZZ$ and then extend to $\RR$: the functions on $\ZZ$
$$
\one_\Set{0, 1, \ldots, m-1} \text{ and } \one_\Set{0, 1, \ldots, n-1}
$$
tile $\ZZ$ with the lattices $\Lambda_1, \Lambda_2 \subseteq \ZZ$ respectively. Taking their convolution in $\ZZ$
$$
g = \one_\Set{0, 1, \ldots, m-1} * \one_\Set{0, 1, \ldots, n-1}
$$
we obtain a function on $\ZZ$ that tiles $\ZZ$ with both $\Lambda_1$ and $\Lambda_2$ and 
has
$$
\supp g = \Set{0, 1, \ldots, m+n-2}.
$$
Defining now (with a slight abuse of notation)
$$
f_2 = g * \one_{[0, 1]}
$$
we obtain a function $f_2$ on $\RR$ that tiles $\RR$ with both $\Lambda_1$ and $\Lambda_2$ and has support
$$
\supp f_2 = [0, m+n-1],
$$
thus improving by 1 on the diameter of $f_1$.
\end{remark}

\subsection{Common tiles whose support has small volume}\label{sec:small-volume}
Another measure of smallness of the support is its volume. Can we construct a common tile $f$ for the lattices $\Lambda_i$ such that $\Abs{\supp f}$ is small?

In the case of $f$ given by \eqref{convolution} it is clear that
$$
\supp{f} = D_1 + D_2 + \ldots + D_N.
$$
To keep things concrete let us assume that all $\Abs{D_i} = 1$ in \eqref{convolution} (unimodular lattices). Then the Brunn-Minkowski inequality \cite{gardner2002brunn} says that
$$
\Abs{\supp f} = \Abs{D_1+\cdots+D_N} \ge \left( \Abs{D_1}^{1/d}+\cdots+\Abs{D_N}^{1/d}\right)^d \ge N^d.
$$
This lower bound
$$
\Abs{\supp f} \ge C N^d
$$
clearly holds also for functions of the form
\beql{convolution-soft}
f = f_1 * f_2 * \cdots * f_N,\ \ \ f_i \ge 0,
\eeq
where for all $i=1,2,\ldots,N$ we assume that the \textit{nonnegative} function $f$ tiles with $\Lambda_i$.

We have proved:
\begin{theorem}\label{th:volume-lb}
For any collection of lattices $\Lambda_1, \ldots, \Lambda_N$ in $\RR^d$ of volume at least 1 and any common tile $f$ for them of the form
$$
f = f_1 * f_2 * \cdots * f_N,\ \ \ f_i \ge 0,
$$
with $f_i$ tiling with $\Lambda_i$, we have
$$
\Abs{\supp f} \ge N^d.
$$
\end{theorem}

But when the functions $f$ are signed (or complex) we only have
$$
\supp{f} \subseteq \supp{f_1}+\cdots\supp{f_N},
$$
not necessarily equality, which brings us to the next question.
\begin{question}
If $f$ is given by \eqref{convolution-soft}, is it true that
\beql{volume-lb}
\Abs{\supp f} \ge C N^d?
\eeq
\end{question}

If one requires that the lattices $\Lambda_1, \Lambda_2, ..., \Lambda_N\subset \mathbb{R}^d$ have the same volume, say $1$, and the sum $\Lambda_1^{\ast}+ \Lambda_2^{\ast}+ ...+ \Lambda_N^{\ast}$ of their dual lattices is direct, then, by \cite[Theorem 2]{kolountzakis1997multi}, they possess a measurable common almost fundamental domain $E$ (generally unbounded). In this case, $|E|=\vol(\Lambda_i)=1$. So then one can take $f=\one_E$, which tiles with all $\Lambda_i$, $i=1,2, ..., N$, with $|\supp f|=|E|=1$.

Motivated by the previous observation, but now dropping the equal volume assumption, we ask the following:

\begin{question} Consider the lattices $\Lambda_1, \Lambda_2, ..., \Lambda_N$, with $\frac{1}{2}\leq \vol(\Lambda_i)\leq 2$. Is there a function $f$ that tiles with all $\Lambda_i$, such that
	\[|\supp f|=o(N^d)?
	\]
\end{question}

In the simplest case in dimension $d=1$, and for two lattices only, a basic question seems to be to ask if the convolution \eqref{convolution-soft} is best in terms of the length of the support. Here we can offer a simple lower bound assuming a nonnegative function.

\begin{theorem}\label{th:1d-lb}
Suppose the nonnegative $f:\RR\to\RR^{\ge 0}$ is measurable and tiles with both $\Lambda_1 = \ZZ$ and with $\Lambda_2 = \alpha \ZZ$, where $\alpha \in (0, 1)$:
\beql{tiling-with-two}
\sum_{n \in \ZZ} f(x-n) = 1,\ \ \ \sum_{n \in \ZZ} f(x-n\alpha) = \frac{1}{\alpha},\ \ \text{for almost every $x \in \RR$}.
\eeq
Then
\beql{1d-lb}
\Abs{\supp f} \ge \Ceil{\frac{1}{\alpha}} \alpha \ge 2\alpha.
\eeq
\end{theorem}

\begin{remark}
If we assume the first equation in \eqref{tiling-with-two} then the constant in the second equation is forced to be $1/\alpha$. This is because $\int f = 1$ (from the first equation), so repeating $f$ at a set of translates of density $1/\alpha$ will give a constant (assuming it tiles) at that level.
\end{remark}

\begin{remark}
Notice that if $\alpha$ is just a little less than 1 then \eqref{1d-lb} gives a lower bound of $2\alpha$, which shows that the convolution $\one_{[0, 1]}*\one_{[0, \alpha]}$ is almost optimal in this case, having support of size $1+\alpha$.

But if, on the other hand, $\alpha$ is just over $1/2$ then the lower bound is just over 1 but the convolution upper bound is just over $3/2$, a considerable gap.
\end{remark}

\begin{proof}
From the first equation in \eqref{tiling-with-two} it follows that $f(x) \le 1$ for almost every $x$. For the second equation to be true it therefore follows that for almost every $x \in \RR$ there are at least $\Ceil{1/\alpha}$ different values of $n \in \ZZ$ such that $f(x-n\alpha)>0$.
Using this for almost all $x \in [0, \alpha)$ (which ensures that for different $x$ the locations $x-n\alpha$ are also different) gives \eqref{1d-lb}.
\end{proof}

\begin{question}
What is the least possible length of the support of $f$ for a nonnegative $f$ that tiles with both $\ZZ$ and $\alpha\ZZ$?
\end{question}

\begin{remark}
\textit{(Added in revision.)}
In \cite{etkind2022functions} the authors obtain sharp results on the measure of the support of a function which tiles the real line simultaneously by translation with respect to two arithmetic progressions $\alpha\ZZ$ and $\beta\ZZ$. In particular if $f$ is nonnegative, then for $\alpha, \beta$ linearly independent over the rationals the smallest measure of the support is $\alpha+\beta$, which is attained if (but not only if) $f$ is the convolution tile $\one_{[0, \alpha]} * \one_{[0, \beta]}$. On the other hand for rationally dependent $\alpha, \beta$ the smallest measure of the support is less than $\alpha+\beta$.
\end{remark}

\subsection{Allowing for lattices with many relations}\label{sec:many-relations}

If we have $N$ lattices
$$
\Lambda_1, \ldots, \Lambda_N \subseteq \RR^d
$$
we can find a function that tiles with them all, namely the function $f$ in \eqref{convolution}. If our lattices are assumed to each have a fundamental domain bounded by $\sim 1$ then $\diam\supp f = O(N)$, and this cannot be improved for functions $f$ arising from \eqref{convolution}. We show here that we can choose the lattices $\Lambda_j$ so that a common tiling function exists which is much more tight than that, tighter even than what Theorem \ref{kw-lb} imposes. Of course our lattices will not satisfy the genericity condition \eqref{no-intersection} of Theorem \ref{kw-lb}, but will satisfy a lot of relations (their intersection will be a large lattice, in terms of density).

Fix a large prime $p$ and consider the group $\ZZ_p^d$. Any nonzero element $g$ of this group generates a cyclic subgroup of order $p$. It follows that $\ZZ_p^d$ has
$$
\frac{p^d-1}{p-1} \sim p^{d-1} =: N
$$
different cyclic subgroups. For each such subgroup $G$, which we now view as a subset of $\Set{0, 1, \ldots, p-1}^d$, consider the lattice
$$
\Lambda_G = (p\ZZ)^d + G,
$$
which contains the lattice $\Lambda = (p\ZZ)^d$ and has volume
$$
\vol \Lambda_G = \frac{\vol (p\ZZ)^d}{\Abs{G}} = p^{d-1} = N.
$$
The function $f = \one_{[0, p)^d}$, $[0,p)^d$ being a fundamental domain of $\Lambda$, tiles with $\Lambda$ and, therefore, with any larger group, so $f$ is a common tile of all $\Lambda_G$.

In order to make the volume of the $\Lambda_G$ equal to 1 we shrink everything by $N^{1/d}$:
$$
\Lambda_G' = N^{-1/d} \Lambda_G,\ \ \ f'(x) = f(N^{1/d} x).
$$
So we have $\sim N$ lattices $\Lambda'_G$ of volume 1 and a common tile $f'$ for them with
$$
\diam \supp f' = \diam\supp f \cdot N^{-1/d} = \sqrt{d} \, p N^{-1/d} = \sqrt{d}\, N^{\frac{1}{d-1}-\frac{1}{d}} = \sqrt{d} \, N^{\frac{1}{d(d-1)}}.
$$

We have proved:
\begin{theorem}\label{th:many-relations}
In dimension $d\ge 2$ and for arbitrarily large $N$ we can find $N$ lattices of volume $1$ and a common tile $f$ for them with
$$
\diam \supp f = O_d\left( N^{\frac{1}{d(d-1)}} \right),
$$
and, consequently, with
$$
\Abs{\supp f} = O_d\left(N^{\frac{1}{d-1}}\right).
$$
\end{theorem}

\begin{question}
Derive a lower bound for $\diam \supp f$, for $f$ tiling with $\Lambda_1, \ldots, \Lambda_N \subseteq \RR^d$ and with $f \ge 0$ (or just $\int f > 0$) under no algebraic conditions for the lattices $\Lambda_j$, assuming only that $\vol\Lambda_j \sim 1$.
\end{question}

The construction that we used to prove Theorem \ref{th:many-relations} gives nothing in dimension $d=1$. Yet, we can prove that, if we allow relations among the lattices, we can achieve $\diam\supp f = o(N)$ in dimension 1 as well.

Let us start by defining
$$
\lambda_j = \frac{1}{N+j},\ \ \ \Lambda_j = \lambda_j\ZZ, \ \ \ (j=1, 2, \ldots, N).
$$
We will first construct a function $f$ which tiles with all the $\Lambda_j$, $j=1, 2, \ldots, N$, such that
$$
\diam\supp f = o(1).
$$
The Fourier transform of such an $f$ must vanish on the dual lattices
$$
\Lambda_j^* = \lambda_j^{-1}\ZZ = (N+j)\ZZ,\ \ \ (j = 1, 2, \ldots, N)
$$
except at 0.
Write
$$
U = \bigcup_{j=1}^N (N+j)\ZZ \ \setminus\Set{0}.
$$
By a result of Erd\H os \cite{erdos1935note} $U$, the set of integers which are divisible by one of the integers in $\Set{N+1, N+2, \ldots, 2N}$, has density tending to 0 with $N$. Tenenbaum \cite{tenenbaum1980lois} has given the estimate that this density is at most
\beql{ten}
\frac{1}{\log^{\delta - o(1)} N},
\eeq
where $\delta = 0.086071\cdots$ is an explicit constant.

It is an important result of Beurling \cite{beurling1989collected} that if $\Lambda$ is a uniformly discrete set of real numbers of upper density $\rho$ then for any $\epsilon>0$ we can find a continuous function $f$, not identically zero, supported by the interval $[0, \rho+\epsilon]$ such that $\ft{f}(\lambda) = 0$ for all $\lambda \in \Lambda$. We can even ask that $\ft{f}(0) = 1$ if $0 \notin \Lambda$. By Tenenbaum's estimate \eqref{ten} we can take $\rho = \log^{-\delta+o(1)}N$ and the set $U$, being a set of integers and thus uniformly discrete, satisfies the assumptions of Beurling's theorem, so there is a function $f$ supported in the interval $[0, \log^{-\delta+o(1)}N]$, with integral 1, such that $\ft{f} = 0$ on $U$. It follows that $f$ tiles with all $\Lambda_j$.

We now scale by a factor of $N$
$$
f'(x) = f(x / N),\ \ \ \Lambda_j' = N \Lambda_j,\ \ \ \diam\supp f' = O(N \log^{-\delta+o(1)}N)
$$
and obtain the first half of the following theorem.
\begin{theorem}\label{th:many-relations-1d}
We can find $N$ lattices $\Lambda_j \subseteq \RR$ of with $\vol\Lambda_j \sim 1$ and a function $f$ with $\int f > 0$ and supported in an interval of length
$$
\frac{N}{\log^{\delta-o(1)}N)}
$$
which tiles with all $\Lambda_j$.

Furthermore, for any $\epsilon > 0$ any such function $f$ must have
$$
\diam\supp f \gtrsim_\epsilon N^{1-\epsilon}.
$$
\end{theorem}

Arguing similarly we can also prove the lower bound for $\diam\supp f$ in Theorem \ref{th:many-relations-1d}. If we assume that $f$ tiles with all $\Lambda_j = \lambda_j \ZZ$, with, say, $1 \le \lambda_j \le 2$, $j=1, 2, \ldots, N$, then $\ft{f}$ vanishes on
$$
\bigcup_{j=1}^N \lambda_j^{-1}\ZZ\ \ \setminus \Set{0}.
$$
If this set is large then Jensen's formula implies that $\diam\supp f$ is also large.
It was proved in \cite[Theorem 1.1, special case $\ell= n$]{gilboa2014union} that, for any $\epsilon > 0$, the above union of arithmetic progressions contains at least $c_\epsilon N^{2-\epsilon}$ points in $[0, 2N]$. By Jensen's formula then we have $\diam\supp f \gtrsim_\epsilon N^{1-\epsilon}$ and this completes the proof of Theorem \ref{th:many-relations-1d}.

\begin{question}
Can we ensure $f \ge 0$ in the first half of Theorem \ref{th:many-relations-1d}?
\end{question}

\section{The common fundamental domain problem without measurability}\label{sec:no-measurability}

In \cite{kolountzakis1997multi} the following theorem was proved in \S 3.2.
\begin{theorem}\label{multi-fd}
If $\Lambda_0, \ldots, \Lambda_n$ are lattices in $\RR^d$ {\em of the same volume} and with the sum $\Lambda_0 + \Lambda_1 + \cdots + \Lambda_n$ being direct then there is a bounded common fundamental domain $F$ for all these lattices.
\end{theorem}

\begin{remark}
No measurability is claimed for $F$ and the set $F$ constructed in \cite{kolountzakis1997multi} is not measurable.
\end{remark}

The question was left open in \cite{kolountzakis1997multi} whether the equal volume assumption was necessary. This assumption is obviously necessary if we ask for a measurable tile as the volume of the tile equals the volume of each lattice it tiles with. But there is no a priori reason for this requirement to hold if we cannot measure volumes, as with tiles that are not necessarily measurable. But, we show here that, indeed it is necessary, by giving a pair of lattices, with no common elements, which have different volumes and have no bounded common fundamental domain.
\begin{theorem}\label{th:equal-volume}
Let $\Lambda_1 = \ZZ^d$ and $\Lambda_2 = \alpha\ZZ^d$, with $\alpha$ irrational and $d \ge 1$. Then there is no bounded set $F \subseteq \RR^d$ which consists of exactly one representative from each coset of each $\Lambda_i$, $i=1,2$.
\end{theorem}
\begin{proof}
We give the proof in dimension $d=1$ for clarity, as it is essentially the same for all $d$. Without loss of generality we take $\alpha > 1$.

As explained in \cite[Proof of Theorem 1]{kolountzakis1997multi} it is enough to show that $\Lambda_1$ and $\Lambda_2$ do not have a bounded common fundamental domain in the group
$$
G = \Lambda_1 + \Lambda_2 = \Set{m+n\alpha: m, n \in \ZZ}.
$$
Suppose $F$ is just such a bounded common fundamental domain
$$
F = \Set{m_i - n_i \alpha: i=1, 2, \ldots} \subseteq [-M, M].
$$
For $F$ to be a fundamental domain it must contain precisely one point in the set $k+\alpha\ZZ$ and one point in the set $\ZZ+k\alpha$, for all $k \in \ZZ$. By renumbering then we can write
$$
F = \Set{m - n_m\alpha: m\in\ZZ}.
$$
Let now $R>0$ be large and consider the following set of values for $m$:
\beql{m-range}
-R \le m \le R.
\eeq
For such values of $m$ and from the presumed bound
$$
\Abs{m-n_m \alpha} \le M
$$
we have the bounds
\beql{n-range}
-\frac{R+M}{\alpha} \le n_m \le \frac{R+M}{\alpha}.
\eeq
As $R\to\infty$ the number of values of $m$ allowed by \eqref{m-range} are $\sim 2R$ in number. The number of the corresponding values of $n$ allowed by \eqref{n-range} is $\sim 2R/\alpha$, which is strictly smaller if $R$ is large since we have taken $\alpha > 1$. This is a contradiction as there must be exactly one $n_m$ for each $m$ and all the $n_m$'s are different.
\end{proof}

\section{The problem in finite abelian groups}\label{sec:finite-groups}

Suppose we have a finite abelian group $G$ and two subgroups $G_1, G_2$ of the same index
$$
n = [G:G_1] = [G:G_2].
$$
We can ask whether we can find a common fundamental domain for $G_1, G_2$ in $G$, i.e. a set of $n$ elements
$$
F = \Set{g_1, \ldots, g_n}
$$
which tiles with both $G_1$ and $G_2$. This always exists, even in the non-abelian case, if properly defined, see e.g.\ \cite{button2014}.

If we drop the equal index assumption we can still ask for a \textit{function} $f$ defined on $G$ which tiles with both subgroups:
\beql{group-tiling}
\forall x \in G:\ \ \sum_{g_1 \in G_1} f(x-g_1) = \Abs{G_1},\ \ \sum_{g_2 \in G_2} f(x-g_2) = \Abs{G_2}.
\eeq
The question we are interested in here is:
\begin{quote}
Given $G, G_1, G_2$ how small can the size of the support of $f$ be?
\end{quote}

Under the assumption of equal index the answer to the above question is that, since we can find a common fundamental domain of $G_1, G_2$ in $G$ \cite{kolountzakis1997multi}, the size of the support of $f$ can be as small as $[G:G_1] = [G:G_2]$. Of course it cannot be smaller than that. But once we drop the equal index assumption then the only general construction we know is the convolution of the indicator functions of the fundamental domains $D_i$ of $G_i$:
\beql{conv-fd}
f = c \one_{D_1} * \one_{D_2},
\eeq
which gives, with $c = \frac{\Abs{G}}{\Abs{D_1} \cdot \Abs{D_2}}$,
$$
f*\one_{G_1} = c \one_{D_2} * \one_{D_1}* \one_{G_1} = c \one_{D_2} * \one_G = c \Abs{D_2} = \Abs{G_1},
$$
and similarly for $f*\one_{G_2}$.
But the support of $f$ in \eqref{conv-fd} can be quite large, a priori as large as $\Abs{D_1}\cdot \Abs{D_2}$.

From now on we restrict our functions $f$ to be nonnegative and normalized as shown in \eqref{group-tiling}. We also usually restrict our tiles to be nonnegative functions.

\begin{definition}
If $G_1, G_2$ are subgroups of the finite abelian group $G$ we write
$$
S_{G_1, G_2}^G = \min\Set{\Abs{\supp f}: \text{where } f:G \to \RR^{\ge 0}, f*\one_{G_1} = \Abs{G_1},\ f*\one_{G_2} = \Abs{G_2}}.
$$
\end{definition}

It is always the case that
$$
S_{G_1, G_2}^G \ge [G : G_i],\ \ \ (i=1,2).
$$
Observe that if $G_1, G_2$ have a common fundamental domain $F$ in $G$ then
$$
S_{G_1, G_2}^G = \Abs{F}.
$$

The following result says that we can always restrict our study to the case of $G_1 \cap G_2$ being trivial. In other words we may assume from now on that $G$ is the direct sum of $G_1$ and $G_2$.

\begin{theorem}\label{th:no-intersection}
If $G_1, G_2 \subseteq G$ are finite abelian groups and
$$
\Gamma = G/(G_1 \cap G_2),\ \  \Gamma_i = G_i/(G_1 \cap G_2),\ \ \ (i=1,2)
$$
then
\beql{S-quotient}
S_{G_1, G_2}^G = S_{\Gamma_1, \Gamma_2}^\Gamma.
\eeq
\end{theorem}

\begin{proof}
In what follows $\Gamma$ or $\Gamma_i$ may denote either the quotient group or an arbitrary but fixed set of coset representatives of the subgroup $G_1 \cap G_2$.

Let $f:G \to \RR^{\ge 0}$ satisfy
$$
\Abs{G_i} = f*\one_{G_i} = (f * \one_{G_1 \cap G_2}) * \one_{\Gamma_i},\ \ \  (i=1,2).
$$
If $\supp f$ is minimal it follows from the above representation that $\supp f$ has at most one point in each $(G_1\cap G_2)$-coset, as we can always collect all the ``mass'' of $f$ contained in one coset to one point of the coset without affecting $f * \one_{G_1 \cap G_2}$.

Define $F:\Gamma \to \RR^{\ge 0}$ by
$$
F(\gamma) = \frac{1}{\Abs{G_1 \cap G_2}} \sum_{g \in G_1 \cap G_2} f(\gamma + g) =
 \frac{1}{\Abs{G_1 \cap G_2}} f*\one_{G_1 \cap G_2}(\gamma).
$$
It follows, under the assumption that $\Abs{\supp f}$ is minimal, that
$$
\Abs{\supp F} = \Abs{\supp f}.
$$
We also have, for $i=1,2$,
$$
F*\one_{\Gamma_i} = \frac{1}{\Abs{G_1 \cap G_2}} f*\one_{G_1 \cap G_2}*\one_{\Gamma_i} = \frac{1}{\Abs{G_1 \cap G_2}} f*\one_{G_i} = \frac{\Abs{G_i}}{\Abs{G_1 \cap G_2}}  = \Abs{\Gamma_i}.
$$
It follows that
$$
S_{\Gamma_1, \Gamma_2}^\Gamma \le S_{G_1, G_2}^G.
$$
To prove the reverse inequality we start with a function $F:\Gamma\to\RR^{\ge 0}$ satisfying
$$
F*\one_{\Gamma_i} = \Abs{\Gamma_i},\ \ \ (i=1,2)
$$
and define $f:G \to \RR^{\ge 0}$ by taking $f(x)$ to be
$$
\Abs{G_1 \cap G_2} \cdot F(x + G_1\cap G_2)
$$
at precisely one point $x$ in each $(G_1 \cap G_2)$-coset, and in all other points of the coset we take it to be 0.

It follows that $\Abs{\supp f} = \Abs{\supp F}$,
and we also have, for $i=1,2$, and viewing $F$ as a function on $G$ (constant on $(G_1 \cap G_2)$-cosets),
$$
f*\one_{G_i} = f*\one_{G_1 \cap G_2} * \one_{\Gamma_i} = 
 \Abs{G_1 \cap G_2} \cdot F* \one_{\Gamma_i} = \Abs{G_1 \cap G_2} \cdot \Abs{\Gamma_i} = \Abs{G_i}.
$$
This concludes the proof of the reverse inequality and the Theorem.

\end{proof}

Let $G = G_1 \times G_2$ from now on. The following theorem is the best understood case.

\begin{theorem}\label{th:multiple}
If $G=G_1 \times G_2$ and $\Abs{G_1}$ divides $\Abs{G_2}$ then
\beql{multiple}
S_{G_1, G_2}^G = [G:G_1] = \Abs{G_2}.
\eeq
\end{theorem}

\begin{proof}
Enumerate the two subgroups arbitrarily as
$$
G_1 = \Set{g^1_1, \ldots, g^1_m},\ \ G_2 = \Set{g^2_1, \ldots, g^2_n},
$$
with $n = k m$, $k \ge 1$.
Take $f = \Abs{G_1} \one_F$ where the set $F$ is constructed by taking sums of the elements of $G_2$ with the ``corresponding'' elements of $G_1$ (but the elements of $G_1$ have to be repeated $k$ times each). So $F$ consists of the sums
\begin{align*}
 &\ g^1_1+g^2_1,\  g^1_2+g^2_2, \ldots, g^1_m+g^2_m,\\
 &\ g^1_1+g^2_{m+1},\  g^1_2+g^2_{m+2}, \ldots, g^1_m+g^2_{m+m}\\
 &\ g^1_1+g^2_{2m+1}, \ldots, g^1_m+g^2_{2m+m}\\
 &\ \ldots
\end{align*}
It is easy to see that $F$ is a fundamental domain for $G_1$ and that
$$
f*\one_{G_1} = \Abs{G_1},\ \ \ f*\one_{G_2} = k \Abs{G_1} = \Abs{G_2}.
$$
Finally $\Abs{\supp f} = \Abs{G_2}$, which proves the Theorem as we always have $S_{G_1, G_2}^G \ge \Abs{G_2}$.
\end{proof}

It is clear now that in studying the problem in the group $G = G_1 \times G_2$ the group structure is irrelevant and, writing $m = \Abs{G_1}$ and $n = \Abs{G_2}$, the problem is to find a nonnegative real matrix
$$
A \in (\RR^{\ge 0})^{m \times n}
$$
with row sums all equal to $n$ and column sums all equal to $m$, and with as small a support (non-zero entries) as possible.

Matrices of this type or, rather, the matrices $\frac{1}{mn}A$, and, more generally, multivariate distributions with uniform marginals, are called {\em copulas} in statistics and have been studied extensively \cite{nelsen2007introduction}. They can be used to ``isolate'' the marginal distributions of a general multivariate distribution from the dependence part of distribution.
\begin{definition}
Write $A(m, n)$
for the set of all such $m \times n$ matrices (with nonnegative entries and all row sums equal to $n$ and all column sums equal to $m$) and
$$
S(m, n) = \min\Set{\Abs{\supp A}:\ A \in A(m, n)}.
$$
\end{definition}
In this notation Theorem \ref{th:multiple} says that
$$
S(m, km) = km,\ \ \text{ if } m, k \in \NN.
$$

In what follows we use \cite{caron1996nonsquare}, where the structure of the matrices in $A(m, n)$ of minimal support is described up to permutation of rows and columns (these operations obviously leave $A(m, n)$ unchanged and also do not alter the size of the support of each matrix).

That the situation changes radically if $m$ does not divide $n$ can be seen, for example, by the following.
\begin{lemma}\label{lm:r}
If $k \ge 1$, $1 \le r < m$, then
$$
S(m, km+r) \ge (k+1) m.
$$
If $r=1$ we have
$$
S(m, km+1) = (k+1) m.
$$
\end{lemma}

\begin{proof}
Column sums are equal to $m$, so none of the entries is $> m$. Since row sums are equal to $km+r$ it follows that we have at least $k+1$ non-zero terms in each of the $m$ rows. This shows the lower bound $S(m, km+r) \ge m(k+1)$. To show the next claim we check that the $m \times (km+1)$ matrix
$$
m \underbrace { \left\{ \ \begin{bmatrix}
1 & \overbrace{m  \cdots  m}^k & \cdots & \cdots \\
1 & \cdots & \overbrace{m  \cdots  m}^k & \cdots \\
\\
\ & \cdots & \cdots & \cdots\\
1 & \cdots & \cdots & \overbrace{m  \cdots  m}^k
\end{bmatrix}\right. }_{km+1}
$$
has constant row and column sums.  Since its support has size $(k+1)m$, which is the minimum possible by the first part of the Lemma, we are done.
\end{proof}

\begin{question}
What is the true value of $S(m, km+r)$ when $1<r<m$?
\end{question}

This question has been completely solved in \cite{loukaki2022doubly} since the first version of this paper appeared.

\bibliographystyle{amsplain}
\bibliography{21095fin}

\end{document}